\documentclass[12pt]{article} 

\usepackage{amsmath}
\usepackage{amsthm}
\usepackage{amsfonts}
\usepackage{mathrsfs}
\usepackage{stmaryrd}
\usepackage{setspace}
\usepackage{fullpage}
\usepackage{amssymb}
\usepackage{breqn}
\usepackage{enumitem}
\usepackage{bbold} 
\usepackage{authblk}
\usepackage{comment}
\usepackage{hyperref}
\usepackage{pgf,tikz}
\usepackage{graphicx}
\usetikzlibrary{decorations.pathreplacing,arrows}
\tikzstyle{vertex}=[circle,draw=black,fill=black,inner sep=0,minimum size=3pt,text=white,font=\footnotesize]

\newtheorem{thm}{Theorem}[section]

\newtheorem{prop}[thm]{Proposition}

\newtheorem{corollary}[thm]{Corollary}

\newtheorem{clm}[thm]{Claim}

\newtheorem*{lemma*}{Lemma}
\newtheorem*{proposition*}{Proposition}
\newtheorem*{theorem*}{Theorem}

\newcommand\ex{\ensuremath{\mathrm{ex}}}

\newcommand\cB{{\mathcal B}}

\newcommand\cG{{\mathcal G}}
\newcommand\cH{{\mathcal H}}

\newcommand\cM{{\mathcal M}}
\newcommand\cN{{\mathcal N}}

\newcommand{\ignore}[1]{}

\title{Generalized Tur\'an problems for complete bipartite graphs}
\author{D\'aniel Gerbner$^1$, Bal\'azs Patk\'os$^{1,2}$\\ \small $^1$ Alfr\'ed R\'enyi Institute of Mathematics\\ \small $^2$ Moscow Institute of Physics and Technology}
\date{}

\begin{document}

\maketitle

\begin{abstract}
    For graph $H$, $F$ and integer $n$, the generalized Tur\'an number $ex(n,H,F)$ denotes the maximum number of copies of $H$ that an $F$-free $n$-vertex graph can have. We study this parameter when both $H$ and $F$ are complete bipartite graphs.
\end{abstract}

\section{Introduction}

The most basic question in extremal graph theory is the following: given a forbidden graph $F$, how many edges can an $n$-vertex graph have, if it does not contain $F$ as a subgraph? A natural generalization is if instead of the number of edges, we are interested in the number of copies of another graph $H$. After several sporadic results (see e.g. \cite{gyps,zykov}), the systematic study of this variant (usually referred to as generalized Tur\'an problem) was initiated by Alon and Shikhelman \cite{as}. Since then, it has attracted the attention of several researchers, see e.g. \cite{bmrsz,gerbner2,ggymv,GP2017,myz}.

We denote by $\ex(n,F)$ the largest number of edges that an $n$-vertex $F$-free graph can have, the ordinary Tur\'an number. For two graphs $H$ and $G$, let $\cN(H,G)$ denote the number of (not necessarily induced) subgraphs of $G$ that are isomorphic to $H$, in other words the number of (unlabeled) copies of $H$ in $G$. Let $\ex(n,H,F):=\max\{\cN(H,G): \text{$G$ is an $n$-vertex $F$-free graph}\}$, i.e. $\ex(n,H,F)$ is the largest number of copies of $H$ that an $n$-vertex $F$-free graph can have.

In this paper we study the generalized Tur\'an problem when both graphs $F$ and $H$ are complete bipartite. 
A number of papers have already studied $\ex(n,K_{a,b},K_{s,t})$ for different values of the parameters. Let us start with recalling what is known about the ordinary Tur\'an number $\ex(n,K_{s,t})=\ex(n,K_{1,1},K_{s,t})$. We will assume without loss of generality that $s\le t$.

In the case $s=1$, the problem is trivial. Let us say that a graph $G$ on $n$ vertices is almost $d$-regular, if either every vertex has degree $d$, or $n-1$ vertices have degree $d$ and one vertex has degree $d-1$. It is easy to see that $\ex(n,K_{1,t})=\lfloor (t-1)n/2\rfloor$, with equality for any almost $(t-1)$-regular graph. 

In the case $s=2$, the problem is much more complicated, and exact results are known only for some values of $n$ in the case $t=2$. However, the asymptotics are known: $\ex(n,K_{2,t})=\frac{1}{2}\sqrt{t-1}n^{3/2}+O(n^{4/3})$. The upper bound comes from the general upper bound by K\H ov\'ari, T. S\'os and Tur\'an \cite{kst}, that we will discuss shortly. The lower bound is by a construction of F\"uredi \cite{fur}. It is an algebraic construction that we will not describe in detail, but we will use it later. We will only use one of its well-known properties (besides it being $K_{2,t}$-free): all but $o(n^2)$ pairs of vertices have exactly $t-1$ common neighbors.

Brown \cite{bro} constructed a $K_{3,3}$-free graph with $(1/2+o(1))n^{5/3}$ edges. F\"uredi \cite{fur3} showed that this is the correct asymptotics. 

K\H ov\'ari, T. S\'os and Tur\'an \cite{kst} showed that $\ex(n,K_{s,t})=O(n^{2-1/s})$. It is generally conjectured to be the correct order of magnitude. Besides the above mentioned cases, it is only known if $t\ge (s-1)!+1$. This was shown in \cite{krsz,arsz}, using the so-called projective norm graphs. Another construction in the case $t$ is even larger compared to $s$, is due to Bukh \cite{bukh}, using random polynomials.

The order of magnitude of $\ex(n,K_{a,b},K_{s,t})$ was studied by Alon and Shikhelman \cite{as} in the case $a\le s\le t$ and $a\le b<t$. 
Their results were extended by Ma, Yuan and Zhang \cite{myz} in the case $a<s$, $b\le s$ and $t$ is large enough and by Bayer, M\'esz\'aros, R\'onyai and Szab\'o \cite{bmrsz} in the case $a\le 3$, $s\ge 4$, $a\le b<t$ and $t$ is large enough. Gerbner, Methuku and Vizer \cite{gmv} studied the case $a,s\le b\le t$. Gerbner \cite{gerbner2} determined $\ex(n,K_{1,b},K_{2,2})$ exactly.

We state a simple upper bound from \cite{gmv} that we use later multiple times.

\begin{prop}[Gerbner, Methuku, Vizer \cite{gmv}]\label{gemevi}
If $s\le a\le b\le t$, then $\ex(n,K_{a,b},K_{s,t})=O(n^s)$.
\end{prop}

For sake of completeness, we sketch the proof. First, we can pick $s$ vertices in $\binom{n}{s}$ ways. Those have at most $t-1$ common neighbors, we pick $a$ of them in at most $\binom{t-1}{a}$ ways. Those $a$ vertices have at most $t-1$ common neighbors, we have already picked $s$ of them, now we pick $b-s$ more in at most $\binom{t-1-s}{b-s}$ ways. We counted every copy of $K_{a,b}$ exactly $\binom{b}{s}$ ways.

\vskip 0.1truecm

The authors of \cite{myz} actually proved a very general result (in fact they proved a hypergraph version of this).

\begin{thm}[Ma, Yuan, Zhang \cite{myz}]\label{kinai}
For any graph $H$, if $t$ is large enough compared to $s$ and $H$, then $\ex(n,H,K_{s,t})=\Omega(n^{|V(H)|-\frac{|E(H)|}{s}})$.
\end{thm}

Let us mention one more general lower bound obtained via a random construction that we will refer to several times in the paper.

\begin{thm}[Gerbner, Palmer \cite{GP2017}]\label{randomconstr}
If for two graphs $F$ and $H$ the inequality $|E(F)|>|E(H)|$ holds, then we have  $\ex(n,H,F)=\Omega(n^{|V(H)|-\frac{|E(H)|(|V(F)|-2}{|E(F)|-|E(H)|}})$.
\end{thm}


We will also use a result of Zhang and Ge \cite{zg}.

\begin{thm}[Zhang, Ge \cite{zg}]\label{neww}
For any $t\ge 2m-3\ge 3$, we have $\ex(n,K_m,K_{2,t})=\Theta(n^{3/2})$.
\end{thm}

The above theorem does not fit into our setting, since cliques are counted rather than complete bipartite graphs. However, it clearly implies the same lower bound if we count any $K_{a,b}$ with $a+b\le m$. In fact, in some cases this gives the best known lower bound even if the forbidden graph is $K_{s,t}$ with $2\le s\le t$.

\bigskip

We now define some notation that we use throughout the paper. Let $G+H$ denote the vertex disjoint union of $G$ and $H$, and $kG$ denote the vertex disjoint union of $k$ copies of $G$.
 We say that a graph $G$ is an extremal graph for $\ex(n,H,F)$ if $G$ has $n$ vertices, is $F$-free and contains $\ex(n,H,F)$ copies of $H$. If $n$, $H$ and $F$ are clear from the context, we just say that $G$ is extremal. Let $\overline{K}_{a,b}$ denote the graph we obtain from $K_{a,b}$ by adding all the edges in the partite set of size $a$. Equivalently, it has $a$ vertices connected to every other vertex, and $b$ vertices with no further edges.

Our paper contains several new results and surveys many old ones. We present all results according to the relationship of the four parameters, $a,b,s,t$. We always assume without loss of generality that $a\le b$ and $s\le t$. The case $a<s$ is considered in Section 2 and is divided into several subsections, depending on the relationship of $b,s$, and $t$. We summarize our findings in the following theorem. The proofs will appear in different subsections.

\begin{thm}\label{a<s}\
\begin{itemize}
    \item[(i)]
    If $a<s$, $s\le b$, then we
have $\ex(n,K_{a,b},K_{s,t})=\Theta(n^b)$.
\item[(ii)]
If $a<s< b$ or if $a<s=b=t$, then we
have $\ex(n,K_{a,b},K_{s,t})=(1+o(1))\binom{s-1}{a}\binom{n}{b}$.
\item[(iii)]
If $a<s\le t< b$, then
there exists $n_0=n_0(a,b,s,t)$ such that if $n\ge n_0$, then every extremal $K_{s,t}$-free $n$-vertex graph, i.e. a graph with the most copies of $K_{a,b}$, contains $K_{s-1,n-s+1}$. 
Moreover, if an $n$-vertex $K_{s,t}$-free graph does not contain $K_{s-1,n-s+1}$, then it has at most $ex(n,K_{a,b},K_{s,t})-\Omega(n^{b-1})$ copies of $K_{a,b}$.
\item[(iv)]
Let $2\le a<t<b$ and $n$ large enough. We have $ex(n,K_{a,b},K_{a+1,t})=\cN(K_{a,b},K_{a,n-a})+\Theta(n)$ if $b<a+t$ and at least one of the following assumptions hold.
\begin{itemize}
    \item 
    $b\le 2a$
    \item
    $2a<t$
    \item
    $a+b<2t-1$. 
\end{itemize} 
Otherwise
 $ex(n,K_{a,b},K_{a+1,t})=\cN(K_{a,b},K_{a,n-a})$ holds. 
 \item[(v)]
 If $a+1<s\le t< b$ and $n$ is large enough, then every extremal $n$-vertex $K_{s,t}$-free graph contains $\overline{K}_{s-1,n-s+1}$. Moreover, if an $n$-vertex $K_{s,t}$-free graph does not contain $\overline{K}_{s-1,n-s+1}$, then it has at most $ex(n,K_{a,b},K_{s,t})-\Omega(n^{b-1})$ copies of $K_{a,b}$.
 \item[(vi)]
 Let $a+1<s\le t< b$ and $n$ is large enough. If $a+b\ge s+t$, then $\ex(n,K_{a,b},K_{s,t})=\cN(K_{a,b},\overline{K}_{s-1,n-s+1})$. If $a+b<s+t$, then $\ex(n,K_{a,b},K_{s,t})=\cN(K_{a,b},\overline{K}_{s-1,n-s+1})+\Theta(n)$.
 \item[(vii)]
 Let $G_1$ be a graph obtained from $\overline{K}_{s-1,n-s+1}$ by adding an $(t-1)$-regular graph of girth at least 5 to the $n-s+1$ vertices forming an independent set. Then if $s\le t< b$ and $n$ is large enough, then $\ex(n,K_{1,b},K_{s,t})=\cN(K_{1,b},G_1)$. 
\end{itemize}
\end{thm}

 In Section 3, we briefly consider the case $a=s$. In Section 4, we study the case $a>s$, prove some upper bounds and obtain lower bounds by constructions derived from known constructions for Tur\'an problems for Berge hypergraphs (the definitions will be introduced in Section 4). The following theorem contains our results.

\begin{thm}\label{a<b}\
\begin{itemize}
    \item[(i)]
For any $a< b< t$ and $n$ we have $ex(n,K_{a,b},K_{1,t})\le \frac{n}{a+b}(\binom{t-1}{a}\binom{t-1}{b-1}+\binom{t-1}{b}\binom{t-1}{a-1})$, while for any $a<t$ and $n$ we have $ex(n,K_{a,a},K_{1,t})\le \frac{n}{2a}\binom{t-1}{a}\binom{t-1}{a-1}$. In particular, if $2t-2$ divides $n$, we have
$\ex(n,K_{a,b},K_{1,t})=\cN(K_{a,b},\frac{n}{2t-2}K_{t-1,t-1})$.
\item[(ii)]
If $a<t\le \frac{8a}{7}+1$, then $\ex(n,K_{a,a},K_{1,t})=\cN(K_{a,a},\lfloor \frac{n}{2t-2}\rfloor K_{t-1,t-1}+K_{\lfloor \frac{p}{2}\rfloor,\lceil \frac{p}{2}\rceil})$, where $p=n-(2t-2)\lfloor \frac{n}{2t-2}\rfloor$.
\item[(iii)]
Let $2\le s<a\le b<t$. Then $\ex(n,K_{a,b},K_{s,t})=\Omega(n^{\frac{3}{2}-o(1)})$.
\end{itemize}
\end{thm}
\section{$a<s$}

\subsection{$a<s$, $s\le b$}

In this case, it is easy to determine the order of magnitude of $ex(n,K_{a,b},K_{s,t})$.


\begin{proof}[Proof of Theorem \ref{a<s} (i)]
The lower bound is given by $K_{s-1,n-s+1}$, which contains $\Theta(n^b)$, more precisely, $\binom{n-s+1}{b}\binom{s-1}{a}$ copies of $K_{a,b}$. For the upper bound, observe that if we choose in a $K_{s,t}$-free graph a $b$-set, its vertices have at most $t-1$ common neighbors. Therefore, $\ex(n,K_{a,b},K_{s,t})\le\binom{t-1}{a}\binom{n}{b}$.
\end{proof}

We can extend this result to determine the asymptotics of $ex(n,K_{a,b},K_{s,t})$ in the case $s<b$.


\begin{proof}[Proof of Theorem \ref{a<s} (ii)]
The lower bound is again given by $K_{s-1,n-s+1}$. For the upper bound, let $G$ be a $K_{s,t}$-free graph on $n$ vertices, and choose a $b$-set $B$ of vertices in $G$. If $t\le b$, then those $b$ vertices have at most $s-1$ common neighbors, thus there are at most $\binom{s-1}{a}$ copies of $K_{a,b}$ containing $B$ as the larger partite set, finishing the proof.

If $t>b$, then $B$ might be the larger partite set of several copies of $K_{a,b}$. However, $G$ contains $O(n^s)$ copies of $K_{s,b}$ by the already proven Theorem \ref{a<s} (i). This means that $O(n^s)$ $b$-sets have at least $s$ common neighbors in $G$, let $\cB$ be the family of such $b$-sets. Let $\cB'$ denote the family of the other $(1-o(1))\binom{n}{b}$ $b$-sets of vertices in $G$. Observe that vertices in a $b$-set in $\cB$ have at most $t-1$ common neighbors. Therefore, we have at most $\binom{t-1}{b}|\cB|=O(n^s)$ copies of $K_{a,b}$ with the larger partite set in $\cB$, and at most $(1-o(1))\binom{s-1}{a}\binom{n}{b}$ copies of $K_{a,b}$ with the larger partite set in $\cB'$, finishing the proof.
\end{proof}

In the remaining case $a<s=b<t$, the asymptotic result does not necessarily hold. If $a=1$ and $b=2$, $K_{1,n-1}$ contains $\binom{n-1}{2}$ copies of $K_{1,2}$. However, Gerbner and Palmer \cite{GP2017} showed that in this case the F\"uredi graph gives the asymptotics, i.e. $\ex(n,K_{1,2},K_{2,t})=(1+o(1))\frac{t-1}{2}n^2$.

For other values of $b=s$ in the case $a=1$, recall that if $t\ge (s-1)!+1$, the projective norm graphs \cite{arsz,krsz} are $K_{s,t}$-free with $(1+o(1))\frac{1}{2}n^{2-1/s}$ edges. As projective norm graphs are close to regular, this shows that they contain $(1+o(1))n^s/s!$ copies of $K_{1,b}$, just like $K_{s-1,n-s+1}$. This suggests that the asymptotic result might hold in the case $a=1$, $b=s>2$.

In the case $s\le b$, we still have two possibilities: either $b\le t$, or $t<b$. We will be able to improve the above proposition by partially describing the structure of the extremal graphs and obtaining exact results in the case $b$ is large compared to $t$.

\subsection{$a<s$, $t< b$}

We start by proving Theorem \ref{a<s} (iii) which we restate for convenience.

\begin{thm}\label{largen} If $a<s\le t< b$, then
there exists $n_0=n_0(a,b,s,t)$ such that if $n\ge n_0$, then every extremal $K_{s,t}$-free $n$-vertex graph, i.e. a graph with the most copies of $K_{a,b}$, contains $K_{s-1,n-s+1}$. 
Moreover, if an $n$-vertex $K_{s,t}$-free graph does not contain $K_{s-1,n-s+1}$, then it has at most $ex(n,K_{a,b},K_{s,t})-\Omega(n^{b-1})$ copies of $K_{a,b}$.
\end{thm}

\begin{proof}
Recall that the already proven part (ii) of  Theorem \ref{a<s} shows that $ex(n,K_{a,b},K_{s,t})=(1+o(1))\cN(K_{a,b},K_{s-1,n-s+1})$. Let us also recall its proof for the present case $t\le b$: in a $K_{s,t}$-free graph $G$ on $n$ vertices, each of the $\binom{n}{b}$ different $b$-sets of vertices has at most $s-1$ common neighbors, thus it forms the larger partite set in at most $\binom{s-1}{a}$ copies of $K_{a,b}$.

For a subset $S$ of vertices, we denote by $N_G(S)$ the set of its common neighbors in $G$. We say that a $b$-set $B$ is $G$-good if $|N_G(B)|=s-1$ and $B$ is $G$-bad if $|N_G(B)|\le s-2$. By the above, if the number of $G$-bad $b$-subsets is at least $\alpha n^{b-1}$ for some $\alpha=\alpha(a,s,b)>0$, then $\cN(K_{a,b},G)<\binom{s-1}{a}\binom{n}{b}-(\binom{s-1}{a}-\binom{s-2}{a})\alpha n^{b-1}<\cN(K_{a,b},K_{s-1,n-s+1})-\Omega(n^{b-1})$.

Let $S_1,S_2,\dots,S_\ell$ be the $(s-1)$-subsets of vertices of $G$ with $|N_G(S_i)|\ge b$ and let us write $N_i=N_G(S_i)$ and $m_i=|N_i|$. Clearly, the number of $G$-good $b$-sets is at most $\sum_{i=1}^{\ell}\binom{m_i}{b}$. The  following simple claim will be crucial to the proof.

\begin{clm}\label{klem}
Let $B$ be a $b$-set of vertices such that $t\le |N_i\cap B|<b$ holds for some $i=1,2,\dots,\ell$. Then $B$ is $G$-bad. Furthermore, $|N_i\cap N_j|<t$ for any $i\neq j$ and the number of $G$-bad $b$-subsets is at least $\frac{1}{\binom{b}{t}}\sum_{i=1}^\ell\binom{m_i}{t}\binom{n-m_i}{b-t}$.
\end{clm}

\begin{proof}[Proof of Claim]
Suppose to the contrary that $B$ is $G$-good and thus, by definition, there exists $j$ such that $N_G(B)=S_j$. As $|N_i\cap B|<b$ implies $B\not\subseteq N_i$, we obtain that $i\neq j$. Then $|S_i\cup S_j|\ge s$ and therefore $t\le |N_i\cap B|$ implies $G[S_i\cup S_j,N_i\cap B]$ contains a copy of $K_{s,t}$. This contradicts the assumption that $G$ is $K_{s,t}$-free. This contradiction proves the first statement of the claim.

If $|N_i\cap N_j|\ge t$, then any $b$-subset $B$ of $N_i$ that contains at least $t$ elements of $N_i \cap N_j$ and at least one element of $N_i\setminus N_j$ contradicts the already proven part of the claim.

Finally, let us count the pairs $(M,B)$, where $M=N_i$ for some $i$, $B$ is a $b$-set and $|B\cap M|=t$. On the one hand, the number of such pairs is exactly $\sum_{i=1}^\ell \binom{m_i}{t}\binom{n-m_i}{b-t}$. On the other hand, by the already proven part of the claim, the $b$-set $B$ in all such pairs must be $G$-bad, and for every $G$-bad $B$ and $t$-subset $T$ of $B$, there can be at most one $N_i$ with $N_i\cap B=T$. Thus the number of pairs $(M,B)$ is at most $\binom{b}{t}$ times the number of $G$-bad $b$-sets.
\end{proof}

With Claim \ref{klem} in hand, we are ready to prove the theorem. We distinguish three cases according to $m=\max \{m_i: i=1,2,\dots,\ell\}$.

\vskip 0.2truecm

\textsc{Case I} $m<\frac{n}{2}$.

\vskip 0.15truecm

By Claim \ref{klem} and the observation before the claim, we obtain that the number of $G$-good $b$-sets is at most $\max \{\binom{n}{b}-\frac{1}{\binom{b}{t}}\sum_i\binom{m_i}{t}\binom{n-m_i}{b-t};\sum_i\binom{m_i}{b}\}$. The assumption $m<\frac{n}{2}$ implies that $\frac{1}{\binom{b}{t}}\binom{m_i}{t}\binom{n-m_i}{b-t}\ge C\binom{m_i}{b}$ holds for all $i$ with some $C>0$ depending only on $b$ and $t$. As a consequence, the number of $G$-good $b$-sets is at most $\frac{1}{C+1}\binom{n}{b}$ and thus $\cN(K_{a,b},G)<\cN(K_{a,b},K_{s-1,n-s+1})-\Omega(n^{b-1})$.

\vskip 0.2truecm

\textsc{Case II} $\frac{n}{2}<m< n-C^*$ for a suitably chosen $C^*$.

\vskip 0.15truecm

If $|N_i|=m$, then by Claim \ref{klem}, all $b$-sets $B$ not contained in $N_i$ but meeting $N_i$ in at least $t$ points are $G$-bad. Their number is at least $\binom{m}{b-1}(n-m)\ge \alpha n^{b-1}$. This yields $\cN(K_{a,b},G)<\cN(K_{a,b},K_{s-1,n-s+1})-\Omega(n^{b-1})$.
\vskip 0.2truecm

\textsc{Case III} $ n-C^*<m<n-s+1$.

\vskip 0.15truecm

Let $i$ be such that $|N_i|=m$. Let $x\in V\setminus (N_i\cup S_i)$ be arbitrary and let us consider the graph $G_x$ that we obtain from $G$ by removing all edges incident to $x$ and then adding all edges between $x$ and $S_i$. i.e. $G[V\setminus \{x\}]=G_x[V\setminus \{x\}]$ and $N_{G_x}(x)=S_i$. Clearly, $G_x$ is $K_{s,t}$-free, as $G[V\setminus \{x\}]=G_x[V\setminus \{x\}]$ and the degree of $x$ is $s-1$. 

As $x\notin N_i$, there exists at least one vertex $y$ of $S_i$ such that $xy\notin E(G)$ and thus at least $\binom{s-2}{a-1}\binom{m}{b-1}$ copies of $K_{a,b}$ contain $x$ in $G_x$ that do not exist in $G$. On the other hand, $|N_G(x)\cap N_i|<t$ as otherwise $G[N_G(x)\cap N_i,S_i\cup \{x\}]$ would contain a $K_{s,t}$. Therefore we can bound the number of $K_{a,b}$'s in $G$ that contain $x$ but do not exist in $G_x$ anymore as follows:
\begin{itemize}
    \item 
    at most $\binom{n-m+t-1}{b}\binom{n}{a-1}=O(n^{a-1})$ copies in which $x$ is in the part of size $a$ (pick the other vertices of the part of size $a$ arbitrarily and then pick the $b$ vertices of the other part among $V\setminus N_i\cup (N_i\cap N_G(x)) $);
    \item
    $\binom{n-m+t-1}{a}\sum_{j=0}^{t-1}\binom{m}{j}\binom{n-m}{b-j}=O(n^{t-1})$ copies in which $x$ is in the part of size $b$. Indeed, if the part of size $b$ of a $K_{a,b}$ contains at least $t$ vertices from $N_i$, then its part of size $a$ cannot contain any vertex from outside $S_i$, and thus these copies of $K_{a,b}$ are present in $G_x$. Vertices of the part of size $a$ are chosen among the neighbors of $x$.
\end{itemize}
We obtained that $\cN(K_{a,b},G)<\cN(K_{a,b},G_x)-\Omega(n^{b-1})$. Thus, in all the above cases, $\cN(K_{a,b},G)<\ex(n,K_{a,b},K_{s,t})-\Omega(n^{b-1})$. This shows that we can find the extremal graph in the only remaining case, $m=n-s+1$, thus $G$ contains $K_{s-1,n-s+1}.$ 
\end{proof}

With the stability result of part (iii) of Theorem \ref{a<s} in hand, we are ready to prove (iv) which we restate here.

\begin{thm}\label{corol}
Let $2\le a<t<b$ and $n$ large enough. We have $ex(n,K_{a,b},K_{a+1,t})=\cN(K_{a,b},K_{a,n-a})+\Theta(n)$ if $b<a+t$ and at least one of the following assumptions hold.

(i) $b\le 2a$

(ii) $2a<t$

(iii) $a+b<2t-1$. 

Otherwise
 $ex(n,K_{a,b},K_{a+1,t})=\cN(K_{a,b},K_{a,n-a})$ holds. 
\end{thm}

\begin{proof}
 Let $G$ be a $K_{a+1,t}$-free graph on $n$ vertices that contains the most copies of $K_{a,b}$. By Theorem \ref{largen} (part (iii) of Theorem \ref{a<s}), $G$ contains $K_{a,n-a}$ with partite sets $A$ of order $a$ and $B$ of order $n-a$. Observe that for a copy $K$ of $K_{a,b}$, both partite sets $A$ and $B$ intersect $K$ in complete bipartite graphs or independent sets. Let us assume that $K$ is not a subgraph of this $K_{a,n-a}$, then the intersection of $K$ and $B$ is not an independent set.
Let $K[A]$ denote the intersection of $A$ and $K$, more precisely $K[A]$ denotes the subgraph of $K$ induced on the vertices $A\cap V(K)$, and we define $K[B]$ analogously. Let $p\le q$ denote the orders of the partite sets of $K[A]$. Then $K[B]$ is either $K_{a-p,b-q}$ or $K_{a-q,b-p}$. Note that $p=0$ is possible and means that $K[A]$ is an independent set in $K$. However, $q=a$ is impossible.

Let $G'$ denote the subgraph of $G$ inside $B$. Observe that a vertex $v$ has degree at most $t-1$ in $G'$ because otherwise $A\cup \{v\}$ and  neighbors of $v$ would form a copy of $K_{a+1,t}$. This means that $G'$ is $K_{1,t}$-free (and in particular both partite sets of $K[B]$ have order less than $t$). $K[B]$ is a complete bipartite graph $K_{p',q'}$ with $p'\le a$, $q'\le b$, and by the above observation we have $p',q'<t$. By part (i) of Theorem \ref{a<s} we have $ex(|B|,K_{p',q'},K_{1,t})=O(n)$, therefore there are $O(n)$ copies of $K_{p',q'}$ for every $p'$ and $q'$ in $G'$. As they can be extended to copies of $K_{a,b}$ from $A$ in constant many ways, we are done with the proof of the upper bound $\cN(K_{a,b},G)\le \cN(K_{a,b},K_{a,n-a})+O(n)$.

Let us show the stronger upper bound $\cN(K_{a,b},G)\le \cN(K_{a,b},K_{a,n-a})$ in the case $b\ge a+t$. Then we have $b-q\ge t$ and $b-p\ge t$, thus one partite set of $K[B]$ has order at least $t$, a contradiction.

Let us assume that $b<a+t$ and show the lower bounds first. In the case $b\le 2a$ or $2a<t$, we can embed $\lfloor \frac{n-a}{b}\rfloor$ copies of $K_{a,b-a}$ into $B$ vertex-disjointly without creating $K_{a+1,t}$. Indeed, a complete bipartite graph can intersect at most one of these copies, and additionally use vertices of $A$. Thus $K$ is a subgraph of either $K_{2a,b-a}$ or $K_{a,b}$. The second graph does not contain $K_{a+1,t}$, and the first one contains it only if either $b-a\ge t$ (which contradicts our assumption $b<a+t$) or $2a\ge t$ and $b-a\ge a+1$. The assumption $b\le 2a$ implies that $b-a\le a<a+1$.

 Assume now that $a+b<2t-1$ and we will also use that $a>\lfloor \frac{t-1}{2}\rfloor$ (otherwise the previous construction gives the desired lower bound). Then we let  $q=\lfloor \frac{t-1}{2}\rfloor$ and $q'=\lceil \frac{t-1}{2}\rceil$. we embed $K_{a-q,q'}$ into $A$ and $\lfloor \frac{n-a}{b}\rfloor$ vertex disjoint copies of $K_{q,b-q'}$ into $B$. Any copy of $K_{a+1,t}$ in the resulting graph intersects $A$ and $B$ both in complete bipartite graphs and in particular is a subgraph of either $K_{a,b}$ or $K_{q+q',a+b-q-q'}$. Clearly $K_{a,b}$ does not contain $K_{a+1,t}$, and both partite sets of $K_{q+q',a+b-q-q'}$ have order less than $t$, a contradiction.

Let us show now the stronger upper bound $\cN(K_{a,b},G)\le \cN(K_{a,b},K_{a,n-a})$ in the remaining cases, i.e. $t\le 2a<b$ and $a+b\ge 2t-1$.
Consider first the case $p=0$.
Then $K[B]$ either contains a partite set of order $a$ or a partite set of order $b$, but this second possibility contradicts our earlier observation that both parts of $K[B]$ have order less than $t$. Thus we can assume that $K[B]$ contains $K_{a,b-a}$.
Then a copy of $K_{a,b-a}$ in $B$ extends to $K_{2a,b-a}$ with the vertices of $A$. We have $b-a>a$ and $2a\ge t$, thus $K_{a+1,t}$ is a subgraph of $G$.

Consider now the case $p>0$
and still we have $a+b\ge 2t-1$ and $b>2a$. 
If $K[B]$ is $K_{a-q,b-p}$, then we have $K_{2a-q,b-p}$ in $G$, which contains $K_{a+1,b-p}$. 
We have $3b/2>a+b\ge 2t-1$ thus $b-p\ge b-a/2>b-b/4=3b/4\ge t-1/2$, so $K_{a+1,b-p}$ and $G$ contains $K_{a+1,t}$. This contradiction finishes the proof in this case. In particular, if $p=q$, then we are done, thus we can assume that $p<q$.
If $K[B]$ is $K_{a-p,b-q}$, then we have $K_{a-p+q,p+b-q}$ in $G$. Since $a-p+q\ge a+1$, we must have $p+b-q<t$. Then $a-p+q=(a+b)-(p+b-q)\ge 2t-1-(p+b-q) \ge t$. 
We also have $p+b-q\ge b-a\ge a+1$, thus there is a $K_{a+1,t}$ in $G$, a contradiction.
\end{proof}


Observe that if $s>a+1$, then $K_{s-1,n-s+1}$ is not the extremal graph. Recall that $\overline{K}_{a,b}$ denotes the graph we obtain from $K_{a,b}$ by adding all the edges in the partite set of size $a$. With a slight abuse of notation, we still use the expression ``partite set'' for the partite sets of the original $K_{a,b}$. The graph
$\overline{K}_{s-1,n-s+1}$ is $K_{s,t}$-free, as it has only $s-1$ vertices of degree at least $t$. On the other hand, if we take $a$ vertices from the part of size $s-1$ to form an $a$-set, one more vertex from that part and $b-1$ vertices from the other part to form a $b$-set, then they are the partite sets of a $K_{a,b}$ that is not present in $K_{s-1,n-s+1}$. We restate (as a proposition) and prove Theorem \ref{a<s} (v).

\begin{prop}\label{overl}
If $a+1<s\le t< b$ and $n$ is large enough, then every extremal $n$-vertex $K_{s,t}$-free graph contains $\overline{K}_{s-1,n-s+1}$. Moreover, if an $n$-vertex $K_{s,t}$-free graph does not contain $\overline{K}_{s-1,n-s+1}$, then it has at most $ex(n,K_{a,b},K_{s,t})-\Omega(n^{b-1})$ copies of $K_{a,b}$.
\end{prop}

\begin{proof} 
Let $u$ and $v$ be vertices in the smaller partite set in $\overline{K}_{s-1,n-s+1}$. We consider first the number of copies of $K_{a,b}$ containing the edge $uv$. We can take $a-1$ other vertices from the smaller partite set and $b-1$ vertices from the larger partite set to form a $K_{a,b}$ with them, thus there are $\Omega(n^{b-1})$ such copies of $K_{a,b}$.

Let $G$ be an extremal $K_{s,t}$-free graph on $n$ vertices, then by Theorem \ref{largen}, $G$ contains a copy of $K_{s-1,n-s+1}$. Let $G'$ be the subgraph of $G$ on the larger partite set. The maximum degree is at most $t-1$ in $G'$ because of the $K_{s,t}$-free property. We consider the copies of $K_{a,b}$ that contain an edge inside the larger partite set. The subgraph of that $K_{a,b}$ inside the larger partite set is a complete bipartite graph $K_{p,q}$ for some $p,q\le t-1$. By Proposition \ref{gemevi}, the number of copies of $K_{p,q}$ in $G'$  is $O(n)$, since $G'$ is a $K_{1,t}$-free graph. Each copy of $K_{p,q}$ can be extended to a copy of $K_{a,b}$ with vertices from the smaller partite set in constantly many ways.

Therefore, the total number of copies of $K_{a,b}$ containing an edge from the larger partite set is $O(n)=o(n^{b-1})$. If any edge is missing in $G$ from the smaller partite set, we lose $\Omega(n^{b-1})$ copies of $K_{a,b}$, thus $G$ has less copies of $K_{a,b}$ than $\overline{K}_{s-1,n-s+1}$, a contradiction.
\end{proof}

It is easy to characterize the $K_{s,t}$-free graphs containing $\overline{K}_{s-1,n-s+1}$: they are those graphs $G$ that have in the larger partite set a subgraph $G'$ which does not contain any $K_{p,q}$ with $p+q=t+1$, $p\le q>t-s$. Indeed, such a $K_{p,q}$ could be extended  by adding $s-p$ and $t-q$ vertices from the smaller partite set of $\overline{K}_{s-1,n-s+1}$ to its partite sets, altogether $s-p+t-q=s-1$ vertices. If, on the other hand, $G$ contains a $K_{s,t}$, it intersects the larger partite set in a complete bipartite graph $K_{p',q}$ with $p'\le q$. Then $p'+q\ge t+1$ and $q\ge t-s$. Let $p=t+1-q$, then deleting $p'+q-t-1$ vertices from the partite set of $K_{p',q}$ of size $p'$ we obtain $K_{p,q}$ with the desired properties. We are ready to restate and prove Theorem \ref{a<s} (vi).

\begin{prop}
Let $a+1<s\le t< b$ and $n$ is large enough. If $a+b\ge s+t$, then $\ex(n,K_{a,b},K_{s,t})=\cN(K_{a,b},\overline{K}_{s-1,n-s+1})$. If $a+b<s+t$, then $\ex(n,K_{a,b},K_{s,t})=\cN(K_{a,b},\overline{K}_{s-1,n-s+1})+\Theta(n)$.
\end{prop}

\begin{proof}
Let $G$ be an extremal $K_{s,t}$-free graph on $n$ vertices.
By Proposition \ref{overl}, $G$ contains $\overline{K}_{s-1,n-s+1}$ with smaller part $A$ and larger part $B$. Assume $G$ contains a copy of $K_{a,b}$ that is not contained in $\overline{K}_{s-1,n-s+1}$. This copy intersects $B$ in a complete bipartite graph $K_{p,q}$ with $p\le q$. Then we have $p+q+s-1\ge a+b$, because these $p+q$ vertices can be extended with vertices from $A$ to $K_{a,b}$. Also we have $p\le a$. 

We can add $s-p$ vertices of $A$ to the smaller partite set of $K_{p,q}$ and the remaining $p-1$ vertices to the larger partite set to obtain a copy of $K_{s,q+p-1}$. 
If $a+b\ge s+t$, then $p+q-1\ge a+b-s\le t$, a contradiction finishing the proof of the first part of the statement.

Assume now that $a+b<s+t$. For the upper bound, let $G'$ denote the subgraph of $G$ inside the larger partite set. Observe that a vertex $v$ has degree at most $t-1$ in $G'$ because otherwise the smaller partite set would extend $v$ and its neighbors to $K_{s,t}$. This means that $G'$ is $K_{1,t}$-free. Every copy of $K_{a,b}$ that is not present in $\overline{K}_{s-1,n-s+1}$ intersects $G'$ in a complete bipartite set $K_{p,q}$ with $p\le a$, $q\le b$. By Proposition \ref{gemevi}, there are $O(n)$ copies of $K_{p,q}$ for every $p$ and $q$ in $G'$. As they can be extended from $A$ constant many ways, we are done with the upper bound.

For the lower bound, we pick $p \le q$ arbitrarily with $p+q=a+b-s+1$. We place $\lfloor (n-s+1)/(p+q)\rfloor$ vertex-disjoint copies of $K_{p,q}$ into the larger partite set $B$ of a $\overline{K}_{s-1,n-s+1}$. Any copy of $K_{s,t}$ in the resulting graph $G'$ intersects this partite set in a complete bipartite graph on at least $t+1>a+b-s+1$ vertices, which is impossible, thus $G'$ is $K_{s,t}$-free. In addition to the copies of $K_{a,b}$ in $\overline{K}_{s-1,n-s+1}$, $G'$ contains those copies that have one of the $\Theta(n)$ copies of $K_{p,q}$ from $B$, and from the set of $s-1$ vertices of degree $n-1$, $a-p$ vertices are added to the smaller partite set and $b-q=p-a+s-1$ vertices to the larger partite set, to form a $K_{a,b}$. As $a-p+b-q=s-1$, this is doable, finishing the proof.
\end{proof}

Let $G_0$ be an almost $(t-1)$-regular graph on $n-s+1$ vertices with girth at least 5. It is well-known that such a $G_0$ exists if $n-s+1$ is large enough. $G_0$ obviously avoids any $K_{p,q}$ with $p+q=t+1$: it avoids $K_{1,t}$ by the degree condition, and it avoids $K_{2,2}$ by the girth condition.
Let us add $G_0$ to $\overline{K}_{s-1,n-s+1}$ to obtain $G_1$. As $G_0$ is not uniquely determined, neither is $G_1$, but their degree sequence is, which in turn determines $\cN(K_{1,b},G_1)$. We end this subsection by restating and proving part (vii) of Theorem \ref{a<s}.

\begin{prop}\label{starr}
If $s\le t< b$ and $n$ is large enough, then $\ex(n,K_{1,b},K_{s,t})=\cN(K_{1,b},G_1)$.
\end{prop}

\begin{proof}
By Proposition \ref{overl}, the extremal graph $G$ contains a copy of $\overline{K}_{s-1,n-s+1}$. The vertices of the larger part are connected to at most $t-1$ other vertices of the larger part, thus have degree at most $s+t-2$. Moreover, if $n-s+1$ and $t-1$ are both odd, than one of the vertices of the large part is connected to at most $t-2$ other vertices of the larger part, thus has degree at most $s+t-3$. Therefore, $G$ contains $s-1$ vertices of degree at most $n-1$, and depending on parity, either $n-s+1$ vertices of degree at most $s+t-2$, or $n-s$ vertices of degree at most $s+t-2$ and one vertex of degree at most $s+t-3$. As $G_1$ has the same degree sequence with equality everywhere, we are done.
\end{proof}

\subsection{$a<s$, $b<s$}
In this case
Alon and Shikhelman \cite{as} showed the upper bound $O(n^{a+b-ab/s})$. They also showed a matching lower bound in the case\footnote{We note that the conditions are incorrectly stated in their paper \cite{as}.} $a\le b<(s+1)/2$ and $t\ge (s-1)!+1$.
Ma, Yuan, Zhang \cite{myz} extended their result with Theorem \ref{kinai} by getting rid of the condition $b<(s+1)/2$ at the cost of increasing $t$ even further: if $a<s$, $b<s$ and $t$ is large enough, then $\ex(n,K_{a,b},K_{s,t})=\Theta(n^{a+b-ab/s})$. 

Bayer, M\'esz\'aros, R\'onyai and Szab\'o \cite{bmrsz} also showed $\ex(n,K_{a,b},K_{s,t})=\Theta(n^{a+b-ab/s})$ in the case $a\le 3$, $a\le b<s$, $s\ge 4$ and $t\ge (s-1)!+1$. Moreover, their results also determine the order of magnitude in some other cases, even though they do not explicitly state them. In particular, if $a=b=2$, $s=3$, then for $t\ge (s-1)!+1=3$, they showed $\ex(n,K_{2,2},K_{3,t})=\Theta(n^{8/3})$.


Observe that we are in the case containing the ordinary Tur\'an problem for $K_{s,t}$, and indeed the results are similar too: the simple upper bound can be matched by a lower bound in the order of magnitude in the case $t$ is large. In other cases, the best known lower bound on $\ex(n,K_{s,t})$ is $\Omega(n^{2-\frac{1}{s}-\frac{1}{t}})$ by a simple random construction. Theorem \ref{randomconstr} extends that construction to generalized Tur\'an problems and gives $\ex(n,K_{a,b},K_{s,t})=\Omega(n^{a+b-\frac{ab(s+t-2)}{st-ab}})$.


\section{$a=s$}

In the case $a=s$, we can again assume that $b<t$, as otherwise $\ex(n,K_{a,b},K_{s,t})=0$. The simple upper bound $O(n^s)$ from Proposition \ref{gemevi} also holds in this case, and it is sharp if $t$ is large enough by Theorem \ref{kinai}.

If $s=1$, the problem is trivial, the almost $(t-1)$-regular graphs are the extremal graphs.
In the case $s=2$, it is not hard to obtain an asymptotically sharp bound. We use the well-known F\"uredi graph \cite{fur} mentioned in the introduction. Recall that it is a $K_{2,t}$-free graph that gives the asymptotic of $\ex(n,K_{2,t})$, and all but $o(n^2)$ pairs of vertices have exactly $t-1$ common neighbors. Gerbner and Palmer \cite{GP2017} showed that among $K_{2,t}$-free graphs, the F\"uredi graph contains asymptotically the most copies of paths and cycles of any length, thus in particular the most copies of $K_{2,2}$. We can extend this to $K_{2,b}$. 

\begin{prop} If $b>2$, then
$\ex(n,K_{2,b},K_{2,t})=(1+o(1))\binom{t-1}{b}\binom{n}{2}$.
\end{prop}

\begin{proof}
Let $G$ be a $K_{2,t}$-free graph on $n$ vertices. We count the copies of $K_{2,b}$ by picking the two vertices in the smaller part. This can be done $\binom{n}{2}$ ways, and then we pick $b$ of their at most $t-1$ common neighbors.

For the lower bound, observe that in the F\"uredi graph each of the $(1-o(1))\binom{n}{2}$ pairs that have exactly $t-1$ common neighbors forms the smaller part of $\binom{t-1}{b}$ copies of $K_{2,b}$.
\end{proof}

Bayer, M\'esz\'aros, R\'onyai and Szab\'o \cite{bmrsz} showed $\ex(n,K_{4,6},K_{4,7})=\Omega(n^{7/4})$.
For other values of $s$, we again use Theorem \ref{randomconstr}. Here $ab<st$ is always satisfied, and we have $\ex(n,K_{a,b},K_{s,t})=\Omega(n^{a+b-\frac{ab(s+t-2)}{st-ab}})=\Omega(n^{a+b-\frac{b(a+t-2)}{t-b}})$. If $t$ increases, the exponent goes to $a$.


\section{$a>s$}

In the case $a>s$, we can assume that $b<t$, as otherwise $\ex(n,K_{a,b},K_{s,t})=0$. Recall the simple upper bound $O(n^s)$ from Proposition \ref{gemevi}. We first consider the case $s=1$. We start with restating and proving Theorem \ref{a<b} (i).

\begin{prop}\label{elso} For any $a< b< t$ and $n$ we have $ex(n,K_{a,b},K_{1,t})\le \frac{n}{a+b}(\binom{t-1}{a}\binom{t-1}{b-1}+\binom{t-1}{b}\binom{t-1}{a-1})$, while for any $a<t$ and $n$ we have $ex(n,K_{a,a},K_{1,t})\le \frac{n}{2a}\binom{t-1}{a}\binom{t-1}{a-1}$. In particular, if $2t-2$ divides $n$, we have
$\ex(n,K_{a,b},K_{1,t})=\cN(K_{a,b},\frac{n}{2t-2}K_{t-1,t-1})$.
\end{prop}

\begin{proof}
Let $G$ be a $K_{1,t}$-free graph and $v$ be a vertex of $G$. Let us count the number of copies of $K_{a,b}$ containing $v$. We start with those copies where $v$ is in the part of size $a$. We have to choose $b$ neighbors of $v$, at most $\binom{t-1}{b}$ ways, and then $a-1$ neighbors of those vertices, at most $\binom{t-1}{a-1}$ ways. If $b\neq a$, we need to add $\binom{t-1}{a}\binom{t-1}{b-1}$ for the copies of $K_{a,b}$ that have $v$ in the part of size $b$. Observe that in the graph $\frac{n}{2t-2}K_{t-1,t-1}$, every vertex is in this many copies of $K_{a,b}$, finishing the proof.
\end{proof}

Let us discuss briefly the case $2t-2$ does not divide $n$. Assume that in the above proof, a vertex is contained in 
$\binom{t-1}{b}\binom{t-1}{a-1}+\binom{t-1}{a}\binom{t-1}{b-1}$ copies of $K_{a,b}$ if $a\neq b$, or $\binom{t-1}{a}\binom{t-1}{a-1}$ copies of $K_{a,b}$ if $a=b$. Then $G$ contains a $K_{t-1,t-1}$. If on the other hand every vertex is contained in less copies of $K_{a,b}$, and $n$ is large enough compared to $a,b,t$, then $G$ contains less copies of $K_{a,b}$ than $\lfloor \frac{n}{2t-2}\rfloor K_{t-1,t-1}$, thus cannot be extremal. This shows that the extremal graph consists of vertex disjoint copies of $K_{t-1,t-1}$, and a subgraph of order $c$ for some constant $c$ that depends on $a$, $b$ and $t$.

We show how to improve this with a more involved calculation in some cases, to obtain exact results. For simplicity, we will deal only with the case $a=b$.
In the case $t<4$, the only meaningful case is $t=3$, $a=2$, which is trivial, as every copy of $K_{2,2}$ is a connected component in a $K_{1,3}$-free graph.
We restate and prove Theorem \ref{a<b} (ii).

\begin{prop}
If $a<t\le \frac{8a}{7}+1$, then $\ex(n,K_{a,a},K_{1,t})=\cN(K_{a,a},\lfloor \frac{n}{2t-2}\rfloor K_{t-1,t-1}+K_{\lfloor \frac{p}{2}\rfloor,\lceil \frac{p}{2}\rceil})$, where $p=n-(2t-2)\lfloor \frac{n}{2t-2}\rfloor$.
\end{prop}

\begin{proof} Let us assume $t\ge 4$. A simple calculation shows that $\binom{t-2}{a}\binom{t-2}{a-1}+\binom{t-2}{a-1}\binom{t-3}{a-1}<\binom{t-1}{a}\binom{t-2}{a-1}/2$, which we will use later.

We deal first with the case $n< 2t-2$. Consider a $K_{1,t}$-free graph $G$.
Let us delete the edges of $G$ that are not contained in any copy of $K_{a,a}$. We claim that the resulting graph is $K_3$-free. Indeed, a triangle $xyz$ would mean that there is a $K_{a,a}$ with partite sets $X$ and $Y$, $x\in X$, $y\in Y$. Then without loss of generality, $z$ is connected to at most $t/2$ vertices in $X$. The edge $xz$ is also in a $K_{a,a}$, where $z$ is in a partite set $Z$ and $x$ is in a partite set $X'$. Then at least $a-t/2-(n-2a)$ elements of $X'$ are in $Y$, as they are all connected to $z$. This shows that $x$ is connected to those vertices plus the vertices of $Z$, thus, using the assumption $t\le \frac{8a}{7}+1$, the degree of $x$ is at least $a+a-t/2-(n-2a)\ge t$, a contradiction.

It was shown by Gy\H ori, Pach and Simonovits \cite{gyps} that among $K_3$-free graphs on $n$ vertices, a complete bipartite graph contains the most copies of $K_{a,b}$, and in particular $K_{\lfloor \frac{n}{2}\rfloor,\lceil \frac{n}{2}\rceil}$ contains the most copies of $K_{a,a}$, finishing the proof in this case.




Proposition \ref{elso} deals with the case $n=2t-2$, thus we can assume now $n>2t-2$.
Recall that in $K_{t-1,t-1}$, every vertex is contained in $\binom{t-1}{a}\binom{t-2}{a-1}$ copies of $K_{a,a}$
Let $G$ be a $K_{1,t}$-free graph on $n$ vertices and $v$ be one of its vertices. If $v$ is contained in a $K_{t-1,t-1}$, then this is a connected component of $G$. Otherwise, there are two possibilities. One of them is that $v$ has less than $t-1$ neighbors, in which case $v$ is contained in at most $\binom{t-2}{a}\binom{t-2}{a-1}<\binom{t-1}{a}\binom{t-2}{a-1}/2$ copies of $K_{a,a}$. The other possibility is that $v$ has a set $P$ of $t-1$ neighbors, but the vertices of $P$ have at most $t-2$ common neighbors. To pick a $K_{a,a}$ containing $v$, we need to pick $a$ vertices from $P$ ($\binom{t-1}{a}$ ways), and then $a$ of their common neighbors (including $v$). If any $a$-set has at most $t-2$ common neighbors including $v$, then $v$ is in at most $\binom{t-1}{a}\binom{t-3}{a-1}<\binom{t-1}{a}\binom{t-2}{a-1}/2$ copies of $K_{a,a}$.

If an $a$-set $A\subset P$ has a set $Q$ of $t-1$ common neighbors (including $v$), then any other $a$-set $A'\subset P$ has common neighbors only from $Q$, as $A'$ intersects $A$ in a vertex $v'$, and $v'$ has at most $t-1$ neighbors, i.e. only the vertices of $Q$. Therefore, the only way that vertices in $P$ do not have $t-1$ common neighbors is that at least one of the vertices of $P$ is not connected to some vertex of $Q$. Therefore, $v$ is contained in at most $\binom{t-2}{a}\binom{t-2}{a-1}+\binom{t-2}{a-1}\binom{t-3}{a-1}<\binom{t-1}{a}\binom{t-2}{a-1}/2$ copies of $K_{a,a}$.

In both cases, we obtained that every vertex that is not in a $K_{t-1,t-1}$ is contained in at most half as many copies of $K_{a,a}$ as those vertices that are in a $K_{t-1,t-1}$. If there are at least $m\ge 2t-2$ vertices in connected components that are different from $K_{t-1,t-1}$, then we can delete them and add $\lfloor m/(2t-2)\rfloor$ copies of $K_{t-1,t-1}$ to obtain more copies of $K_{a,a}$. Therefore, we can assume that we have $\lfloor n/(2t-2)\rfloor$ copies of $K_{t-1,t-1}$ in $G$, and the extremal graph on the remaining $p$ vertices, finishing the proof.
\end{proof}

We remark that with the same proof, one can obtain a similar bound for $\ex(n,K_{a,b},K_{s,t})$ if $t$ is not much bigger than $a$. However, the additional small graph on the remaining $p$ vertices might be unbalanced.

For larger $s$, we are unable to improve the upper bound. We can improve the trivial linear lower bound using Theorem \ref{neww}, which yields  $\ex(n,K_{a,b},K_{s,t})=\Omega(n^{3/2})$ if $t\ge 2a+2b-3$. For smaller values of $t$,
 Theorem \ref{randomconstr} yields that if $ab<st$, then $\ex(n,K_{a,b},K_{s,t})=\Omega(n^{a+b-\frac{ab(s+t-2)}{st-ab}})$.



Let us show now a connection to hypergraph Tur\'an problems.
A hypergraph $\cG$ is a Berge copy of a graph $G$ (in brief: a Berge-$G$) if its vertex set contains $V(G)$, and there is a bijection between the edges of $G$ and the hyperedges of $\cG$ such that each edge is contained in its image. In other words, we can add additional vertices to the edges of $G$ (arbitrarily) to obtain $\cG$.

Berge hypergraphs were introduced by Gerbner and Palmer \cite{gp1}, extending the well-established notion of hypergraph cycles due to Berge. A hypergraph is linear if any two of its hyperedges share at most one vertex. We say that a hypergraph has girth $k$ if it is linear and its shortest Berge cycle has length $k$. Note that the linearity can be thought of as forbidding Berge cycles of length 2.


\begin{prop} Let $2<a\le b<t$.
Let $\cH$ be an $(a+b)$-uniform hypergraph of girth at least 5. Let $G$ be a graph obtained by placing an arbitrary $K_{a,b}$ into every hyperedge of $\cH$. Then $G$ is $K_{2,t}$-free.
\end{prop}

\begin{proof}
Assume indirectly that there is a copy $K$ of $K_{2,t}$ in $G$ with partite sets $\{u,u'\}$ and $T$. Observe that if an edge $e$ of $G$ is contained in a hyperedge $h$ of $\cH$, then no other hyperedge of $\cH$ contains $e$, hence $e$ must be in the copy of $K_{a,b}$ embedded into $h$. Let $v,v'$ be vertices of $T$. If the four edges $uv,u'v,uv',u'v'$ are contained in four different hyperedges, those hyperedges form a Berge-$C_4$, a contradiction. Otherwise there is a hyperedge $h$ containing at least three of $u,u',v,v'$.

Assume first that $h$ contains $u,u',v$. If there is a vertex $v''\in T\setminus h$, then $uv''$ and $u'v''$ must come from hyperedges $h',h''$ different from each other and from $h$ because of the linearity, and then $h,u,h',v'',h'',u'$ form a Berge triangle, a contradiction. If $h$ contains $T$, then $u,u'$ are in the same partite set of the $K_{a,b}$ embedded into $h$. Then a third vertex $v''\in T$ is in the same partite set (as $T$ does not fit into the other partite set), but then the edge $uv''$ cannot be in $G$, a contradiction.

Assume now that $h$ contains $u,v,v'$, but not $u'$. If the edges $u'v$ and $u'v'$ are in the same hyperedge $h'$, that contradicts the linearity. If $u'v\in h'$ and $u'v'\in h''$ with $h'\neq h''$, then $h,v,h',u',h'',v'$ is a Berge triangle, a contradiction finishing the proof.
\end{proof}

By the above proposition, in the case $2\le s<a\le b<t$, we have that $\ex(n,K_{a,b},K_{s,t})$ is at least the largest size of an $(a+b)$-uniform hypergraph $\cH$ without Berge cycles of length at most $4$. 
As hyperedges of $\cH$ intersect in at most one vertex, it is easy to see that $\cH$ has $O(n^2)$ hyperedges, thus this lower bound is often much weaker than the previous ones. In fact, as Berge cycles of length 4 are forbidden, there are $O(n^{3/2})$ hyperedges in $\cH$ due to a result of Gy\H ori and Lemons \cite{gyl}. 
Lazebnik and Verstra\"ete \cite{lazver} studied $r$-uniform hypergraphs without Berge cycles of length at most 4, and proved the bound $\Omega(n^{4/3})$. Timmons and Verstra\"ete \cite{vt} mention that this can be improved to $n^{\frac{3}{2}-o(1)}$ using a construction of Ruzsa \cite{ruzsa}, thus we have the same lower bound for $\ex(n,K_{a,b},K_{s,t})$ if $2<a\le b<t$. We obtained Theorem \ref{a<b} (iii), which we restate as a corollary.

\begin{corollary}
Let $2<a\le b<t$. Then $\ex(n,K_{a,b},K_{s,t})=\Omega(n^{\frac{3}{2}-o(1)})$.
\end{corollary}

In the rest of this section we describe further similar connections to hypergraph Tur\'an problems. However, they do not improve the bounds we have obtained, due to the lack of results on the corresponding hypergraph Tur\'an problems. Furthermore, we relate $\ex(n,K_{a,b},K_{s,t})$ to the size of some linear hypergraphs, thus we cannot hope for a superquadratic lower bound (recall that our upper bound is $O(n^s)$). For these reasons, we postpone the proofs to the Appendix.

\begin{prop}\label{p1}
Let $2\le s<a\le b$, $p\ge s$ and $t\ge (a+b-2)\binom{s}{2}+(a+b-1)^sp$.
Let $\cH$ be an $(a+b)$-uniform linear hypergraph that does not contain a Berge-$K_{s,p}$. Let $G$ be a graph obtained by placing an arbitrary $K_{a,b}$ into every hyperedge of $\cH$. Then $G$ is $K_{s,t}$-free.
\end{prop}

Unfortunately, there are not many results on linear hypergraphs without Berge-$K_{s,p}$. In the case $s=2$, Timmons \cite{timm} showed a construction with $\Omega(n^{3/2})$ hyperedges for any uniformity (further results were obtained in \cite{gmv2} for uniformity 3). However, this does not improve the bounds obtained using Theorem \ref{neww}.

We can say more at the cost of increasing $t$ (not a big cost, as we did not make any effort to optimize the threshold on $t$). Let us define a third type of graph-based hypergraph. The $r$-uniform expansion $G^{+r}$ of a graph $G$ is the Berge copy where for two edges $e$ and $e'$ of $G$, their images share only the one or zero vertex that $e$ and $e'$ share. In other words, we can add $r-2$ additional vertices to the edges of $G$ to obtain $\cG$, such that every new vertex is added only to one hyperedge.

\begin{prop}\label{p2}
Let $2\le s<a\le b$, $p\ge s$ and $t\ge (p-1)s(s-1)(a+b-2)^2+(a+b-2)\binom{s}{2}+1$.
Let $\cH$ be an $(a+b)$-uniform linear hypergraph that does not contain a $K_{s,p}^{+(a+b)}$. Let $G$ be a graph obtained by placing an arbitrary $K_{a,b}$ into every hyperedge of $\cH$. Then $G$ is $K_{s,t}$-free.
\end{prop}

There are several extremal results concerning expansions, see \cite{MV2016} for a survey. However, much less is known in the linear setting. 

\section{Concluding remarks}

In this paper we studied $\ex(n,K_{a,b},K_{s,t})$ for different values of the parameters, with varying success. We left wide open the question what happens if $1<s<a\le b<t$. None of the upper and lower bounds we obtained seems to be close to the value of $\ex(n,K_{a,b},K_{s,t})$ in this case. 

In Section 2, Theorem \ref{a<s} (iii) and (iv) are stability results for some values of $a,b,s,t$. Very recently, Livinsky \cite{liv} observed that the $K_{2,t+1}$-free constructions of F\"uredi do not contain any $K_{3,3}$. It is not hard to see that the construction with $t=(1+o(1))n^{1/3}$ contains $(1/2+o(1))\binom{n}{3}=(1/2+o(1))ex(n,K_{2,3},K_{3,3})$ copies of $K_{2,3}$ and its structure is very different from that of the extremal graphs.

An obvious next step in future research is to consider complete $r$-partite graphs. Some of our results can be extended to this setting easily. Let us show one example, a generalization of Theorem \ref{a<s} (i).

\begin{prop}
Let $a_1\le\dots\le a_r$ and $s_1\le\dots\le s_r$. If $a_1<s_1$ and $s_i\le a_{i+1}$ for each $i<r$,
then we have $\ex(n,K_{a_1,\dots,a_r},K_{s_1,\dots,s_r})=\Theta(n^{a_2+a_3+\dots+a_r})$.
\end{prop}

\begin{proof}
The lower bound is given by any complete $r$-partite graph with one part of size $s_1-1$ and $r-1$ parts of linear size. For the upper bound, observe that if we choose in a $K_{s_1,\dots,s_r}$-free graph a copy of the complete $(r-1)$-partite graph $K_{a_2,\dots,a_r}$, then its vertices have at most $s_r-1$ common neighbors. Therefore, the number of copies of $K_{a_1,\dots,a_r}$ is at most $\binom{s_r-1}{a_1}$ times the number of copies of $K_{a_2,\dots,a_r}$, which is $\Theta(n^{a_2+a_3+\dots+a_r})$, finishing the proof.
\end{proof}

However, the cases $r\ge 3$ can be much more complicated and introduce new phenomena. In particular, let us consider the smallest such pair of graphs, $K_{1,1,1}=K_3$ and $K_{1,1,2}$. One would think that this would be an easy problem and difficulties arise only for larger values of the parameters. This is far from the truth.
Alon and Shikhelman \cite{as} showed that for any $t$, $\ex(n,K_{1,1,1},K_{1,1,t})=o(n^2)$ but $\ex(n,K_{1,1,1},K_{1,1,t})=n^{2-o(1)}$. Moreover, this is a reformulation of the celebrated Ruzsa-Szemer\'edi 6-3 theorem \cite{rsz}. Gowers and Janzer \cite{gj} generalized this for any $r$. Let $K$ denote the complete $r$-partite graph on $r+1$ vertices, i.e. $K_{1,1,\dots,1,2}$. Then $\ex(n,K_r,K)=o(n^{r-1})$ but $\ex(n,K_r,K)=n^{r-1-o(1)}$.

\bigskip

A possible weakening is to forbid a subgraph of $K_{s,t}$. 
Gerbner, Nagy and Vizer \cite{gnv} showed $\ex(n,K_{2,b},C_{2s})=(1+o(1))\binom{s-1}{2}\binom{n}{b}=(1+o(1))\cN(K_{2,b},K_{s-1,n-s+1})$ (the case $b=2$ was proved in \cite{ggymv}). This result is implied by Theorem \ref{a<s} (ii) if $b\ge s$, and is improved to an exact result by Theorem \ref{a<s} (iv) if $b>s=3$.
Similarly, in the cases where the extremal construction is $K_{s-1,n-s+1}$, then the bounds we obtained for $\ex(n,K_{a,b},K_{s,t})$ also hold for $\ex(n,K_{a,b},F)$ for each bipartite graph $F$ with the property that in any proper two-coloring of $F$, both color classes have order at least $s$ and at most $t$. 

\section{Appendix}

\begin{proof}[Proof of Proposition \ref{p1}]
Assume there is a copy of $K_{s,t}$ in $G$ with smaller partite set $S$ and larger partite set $T$. First we show that the number of vertices of $T$ that are contained in a hyperedge with more than one vertex of $S$ is at most $(a+b-2)\binom{s}{2}$. Indeed, every 2-set in $S$ is contained in at most one hyperedge of $\cH$, and that hyperedge contains $a+b-2$ other vertices. Let $T'$ be an arbitrary set of $(a+b-1)^sp$ other vertices in $T$.

Therefore, there is a $K_{s,(a+b-1)^sp}$ with partite sets $S$ and $T'$, such that for every vertex of $T'$, the edges of this $K_{s,(a+b-1)^sp}$ incident with it each come from distinct hyperedges of $\cH$. Consider a $u\in S$. At most $a+b-1$ vertices of $T'$ are contained in the same hyperedge together with $u$. For each hyperedge, we keep only one such vertex, and delete the rest. This way we keep at least $(a+b-1)^{s-1}p$ vertices from $T'$. We repeat this for every vertex of $S$, and at the end there are at least $p$ vertices remaining. Therefore, we obtain a $K_{s,p}$ where each edge comes from a hyperedge that contains only one vertex from $S$ and only one vertex from $T$. Those hyperedges form a Berge-$K_{s,p}$, a contradiction. 
\end{proof}

We remark that the proof shows a bit more: we do not need to forbid every Berge-$K_{s,p}$, only those where in the bijection defining the Berge-$K_{s,p}$, the image of every edge $uv$ contains only $u$ and $v$ from $V(K_{s,p})$. In other words, we can add additional vertices that are not in $V(G)$ to the edges of $G$ to obtain $\cG$. Equivalently, the trace of a subhypergraph of $\cH$ is exactly a $K_{s,p}$ on a set of $s+p$ vertices. Extremal problems for such hypergraphs have been studied, see e.g. \cite{fl}, but we are not aware of any results in the linear setting.

\begin{proof}[Proof of Proposition \ref{p2}]
Assume there is a copy of $K_{s,t}$ in $G$ with smaller partite set $S=\{v_1,\dots,v_s\}$ and larger partite set $T$. First we show that the number of vertices of $T$ that are contained in a hyperedge with more than one vertex of $S$ is at most $(a+b-2)\binom{s}{2}$. Indeed, every 2-set in $S$ is contained in at most one hyperedge of $\cH$, and that hyperedge contains at most $a+b$ vertices. Let $T^1$ be an arbitrary set of $(p-1)s(s-1)(a+b-2)^2+1$ other vertices in $T$.

Therefore, there is a $K_{s,(p-1)s(s-1)(a+b-2)^2+1}$ with partite sets $S$ and $T^1$, such that for every vertex $u\in T^1$, the edges of this subgraph incident with $u$ each come from distinct hyperedges of $\cH$. Let us pick an arbitrary hyperedge $h$ containing $v_1$ and a vertex $u_1$ of $T^1$ (recall that $v_1\in S$). Let $w$ be another vertex of $h$. By the above, $w\not\in S$. If $w\in T^1$, we delete it from $T^1$. Otherwise, there are at most $s-1$ hyperedges besides $h$ that contain a vertex from $S$ and contain $w$ by the linearity of $\cH$. These $s-1$ hyperedges contain at most $(s-1)(a+b-2)$ vertices from $T^1$; we delete those vertices. This way altogether we deleted at most $(s-1)(a+b-2)^2$ vertices. Then we choose a hyperedge containing $v_2$ and $u_1$, and repeat this procedure. After that we repeat it for $v_3$ and $u_1$, and so on. 

At the end of this we have deleted at most $s(s-1)(a+b-2)^2$ vertices of $T^1$ to obtain $T^2$. Note that by the above procedure, the $s$ hyperedges containing $u_1$ and $v_i$ for $1\le i\le s$ are such that these hyperedges do not share any vertex with any hyperedge that contains a vertex from $S$ and a vertex from $T^2$. We pick another  hyperedge containing $v_1$ and a vertex $u_2$ of $T^2$ and apply the same procedure, and so on. That is we obtain a subset $T^3$ of $T^2$ such that the hyperedges $h_{i,2}$ containing $v_i$ and $u_2$ do not share any vertex with any hyperedge that contains a vertex from $S$ and a vertex from $T^3$. In general, we pick $u_j$ and then $T^{j+1}\subset T^j$ such that that the hyperedges $h_{i,j}$ containing $v_i$ and $u_j$ do not share any vertex with any hyperedge that contains a vertex from $S$ and a vertex from $T^{j+1}$. As in each iteration, we delete at most $s(s-1)(a+b-2)^2$ vertices, after picking $u_1,\dots, u_{p-1}$, we still have a last vertex $u_p$ to pick. 

We are left with a copy $K$ of $K_{s,p}$ in $G$, where one of the partite sets is $S$. Let us denote the other partite set by $P$. For each edge $v_iu_j$ of $K$, if $1\le j\le p-1$ and $1\le i\le s$, then we picked a hyperedge $h_{i,j}$ containing $u_i$ and $v_j$. If $i=p$, we pick the hyperedge containing $u_p$ and $v_i$ assured by the complete bipartite subgraph of $G$ on $S$ and $T^1$. We claim that the hypergraph formed by these hyperedges is a $K_{s,p}^{+(a+b)}$, a contradiction. Indeed, if $i=i'$ or $j=j'$, then, by linearity of $\cH$, $h_{i,j}\cap h_{i',j'}$ consists of $u_i$ or $v_j$. Without loss of generality, we can assume $j<j'$. But then $T^{j+1}\supseteq T^{j'}$ was chosen such that the hyperedge $h_{i,j}$ do not share any vertex with any hyperedge that contains a vertex from $S$ and a vertex from $T^{j+1}$ such as $h_{i',j'}$.
\end{proof}

\end{document}